\documentclass{amsart}
\usepackage{amsmath, amsfonts, amssymb, stmaryrd}

\newtheorem{thm1}{Theorem}
\newtheorem{thm}{Theorem}[section] 
\newtheorem{lemma}[thm]{Lemma}
\newtheorem{prop}[thm]{Proposition}
\newtheorem{cor}[thm]{Corollary}

\theoremstyle{definition}
\newtheorem{remark}[thm]{Remark}

\DeclareMathOperator{\OG}{OG}

\DeclareMathOperator{\QK}{QK}

\DeclareMathOperator{\Lie}{Lie}
\DeclareMathOperator{\Pic}{Pic}

\newcommand{\bP}{{\mathbb P}}

\newcommand{\C}{{\mathbb C}}

\newcommand{\Z}{{\mathbb Z}}

\newcommand{\cO}{{\mathcal O}}

\newcommand{\cZ}{{\mathcal Z}}

\newcommand{\al}{{\alpha}}

\newcommand{\ev}{\operatorname{ev}}
\newcommand{\wt}{\widetilde}

\newcommand{\wb}{\overline}

\newcommand{\bd}{{\bf d}}
\newcommand{\be}{{\bf e}}
\newcommand{\ignore}[1]{}

\newcommand{\Mb}{\wb{\mathcal M}}

\def\N{{\mathbb N}}

\newcommand{\drc}{{d_{\rm rc}}}
\newcommand{\dcl}{{d_{\rm cl}}}

\begin{document}

\title{Rational connectedness implies\\ Finiteness of quantum
  $K$-theory}

\date{May 10, 2013}

\author{Anders~S.~Buch}
\address{Department of Mathematics, Rutgers University, 110
  Frelinghuysen Road, Piscataway, NJ 08854, USA}
\email{asbuch@math.rutgers.edu}

\author{Pierre--Emmanuel Chaput}
\address{Domaine Scientifique Victor Grignard, 239, Boulevard des
  Aiguillettes, Universit{\'e} Henri Poincar{\'e} Nancy 1, B.P. 70239,
  F-54506 Vandoeuvre-l{\`e}s-Nancy Cedex, France}
\email{pierre-emmanuel.chaput@univ-lorraine.fr}

\author{Leonardo~C.~Mihalcea}
\address{Department of Mathematics, Virginia Tech University, 460 McBryde,
  Blacksburg VA 24060, USA}
\email{lmihalce@math.vt.edu}

\author{Nicolas Perrin}
\address{Mathematisches Institut, Heinrich-Heine-Universit{\"a}t,
D-40204 D{\"u}sseldorf, Germany}
\email{perrin@math.uni-duesseldorf.de}


\thanks{The first author was supported in part by NSF grant
  DMS-1205351.}

\thanks{The third author was supported in part by NSA Young
  Investigator Award H98230-13-1-0208.}

\begin{abstract}
  Let $X$ be any generalized flag variety with Picard group of rank
  one.  Given a degree $d$, consider the Gromov-Witten variety of
  rational curves of degree $d$ in $X$ that meet three general points.
  We prove that, if this Gromov-Witten variety is rationally connected
  for all large degrees $d$, then the structure constants of the small
  quantum $K$-theory ring of $X$ vanish for large degrees.
\end{abstract}

\maketitle

\markboth{A.~BUCH, P.--E.~CHAPUT, L.~MIHALCEA, AND N.~PERRIN}
{RATIONAL CONNECTEDNESS IMPLIES FINITENESS}

\section{Introduction}\label{sec:intro}

The (small) quantum $K$-theory ring $\QK(X)$ of a smooth complex
projective variety $X$ is a generalization of both the Grothendieck
ring $K(X)$ of algebraic vector bundles on $X$ and the small quantum
cohomology ring of $X$.  The ring $\QK(X)$ was defined by Givental
\cite{givental:wdvv} when $X$ is a rational homogeneous space and by
Lee \cite{lee:quantum*1} in general.  In this paper we study this ring
when $X$ is a complex projective rational homogeneous space with
$\Pic(X)=\Z$.  Equivalently, we have $X = G/P$ where $G$ is a complex
semisimple algebraic group and $P\subset G$ is a maximal parabolic
subgroup.  The product in $\QK(X)$ of two arbitrary classes $\al,
\beta \in K(X)$ is a power series
\[
\alpha \star \beta = \sum_{d \geq 0} (\alpha \star \beta)_d \, q^d \,,
\]
where each coefficient $(\alpha \star \beta)_d \in K(X)$ is defined
using the $K$-theory ring of the Kontsevich moduli space
$\Mb_{0,3}(X,d)$ of stable maps to $X$ of degree $d$.  For general
homogeneous spaces it is an open problem if this power series can have
infinitely many non-zero terms.  The product $\al\star\beta$ is known
to be finite if $X$ is a Grassmann variety of type A
\cite{buch.mihalcea:quantum}.  More generally, when $X$ is any
cominuscule homogeneous space, it was proved by the authors in
\cite{buch.chaput.ea:finiteness} that all products in $\QK(X)$ are
finite.  Let $d_X(2)$ denote the smallest possible degree of a
rational curve connecting two general points in $X$.  The main theorem
of \cite{buch.chaput.ea:finiteness} states that $(\al \star \beta)_d =
0$ whenever $X$ is cominuscule and $d > d_X(2)$, which is the best
possible bound.

Given three general points $x,y,z \in X$, let $M_d(x,y,z) \subset
\Mb_{0,3}(X,d)$ denote the {\em Gromov-Witten variety\/} of stable
maps that send the three marked points to $x$, $y$, and $z$.  We will
assume that this variety is rationally connected for all sufficiently
large degrees $d$.  Let $\drc$ be a positive integer such that
$M_d(x,y,z)$ is rationally connected for $d \geq \drc$.  We also let
$\dcl$ be the smallest length of a chain of lines connecting two
general points in $X$.  Our main result is the following theorem.

\begin{thm1}\label{thm:main}
  We have $(\al\star\beta)_d = 0$ for all $d \geq \drc + \dcl$.
\end{thm1}

The Gromov-Witten varieties $M_d(x,y,z)$ of large degrees are known to
be rational when $X$ is a cominuscule homogeneous space, an orthogonal
Grassmannian $\OG(m,N)$ for $m \neq \frac{N}{2}-1$, or any adjoint
variety of type different from A or G$_2$.  This was proved in
\cite{buch.mihalcea:quantum} for Grassmannians of type A and in
\cite{chaput.perrin:rationality} in all other cases.
Theorem~\ref{thm:main} therefore establishes the finiteness of quantum
$K$-theory for many new spaces.  The {\em orthogonal Grassmannian\/}
$\OG(m,N)$ is the variety of isotropic $m$-dimensional subspaces in
the vector space $\C^N$ equipped with a non-degenerate symmetric
bilinear form; these varieties account for all spaces $G/P$ where $G$
is a group of type $B_n$ or $D_n$ and $P$ is a maximal parabolic
subgroup.  The variety $X = G/P$ is called {\em adjoint\/} if it is
isomorphic to the closed orbit of the adjoint action of $G$ on
$\bP(\Lie(G))$.

\ignore{
\begin{remark}
  Jason Starr reports that the results of \cite{jong.he.ea:families}
  can be used to prove that the Gromov-Witten varieties $M_d(x,y,z)$
  of large degrees are rationally connected when $X$ is any projective
  rational homogeneous space with $\Pic(X)=\Z$.
\end{remark}
}

\begin{remark}
  We thank Jason Starr for sending us an outline of an argument that
  uses the results of \cite{jong.he.ea:families, jong.starr:low} to
  prove that the Gromov-Witten varieties $M_d(x,y,z)$ of large degrees
  are rationally connected when $X$ is any projective rational
  homogeneous space with $\Pic(X)=\Z$.  As a consequence,
  Theorem~\ref{thm:main} can be applied to all such spaces.  We also
  thank Starr for making us aware of
  \cite[Lemma~15.8]{jong.he.ea:families}.
\end{remark}

\section{Stable maps and Gromov-Witten varieties}

We recall here some notation and results from
\cite{buch.chaput.ea:finiteness}.  Let $X = G/P$ be a homogeneous
space defined by a semisimple complex linear algebraic group $G$ and a
parabolic subgroup $P \subset G$.  Given an effective degree $d \in
H_2(X;\Z)$ and an integer $n \geq 0$, the Kontsevich moduli space
$\Mb_{0,n}(X,d)$ parametrizes the isomorphism classes of $n$-pointed
stable (genus zero) maps $f : C \to X$ with $f_*[C] = d$, and comes
with a total evaluation map $\ev = (\ev_1,\dots,\ev_n) :
\Mb_{0,n}(X,d) \to X^n := X \times \dots \times X$.  A detailed
construction of this space can be found in the survey
\cite{fulton.pandharipande:notes}.

Let $\bd = (d_0,d_1,\dots,d_r)$ be a sequence of effective classes
$d_i \in H_2(X;\Z)$, let $\be = (e_0,\dots,e_r) \in \N^{r+1}$, and set
$|\bd|=\sum d_i$ and $|\be|=\sum e_i$.  We will assume that $e_0>0$
and $e_r>0$, and that $e_i \geq 1+\delta_{i,0}+\delta_{i,r}$ whenever
$d_i=0$.  Let $M_{\bd,\be} \subset \Mb_{0,|\be|}(X,|\bd|)$ be the
closure of the locus of stable maps $f : C \to X$ defined on a chain
$C$ of $r+1$ projective lines, such that the $i$-th projective line
contains $e_i$ marked points (numbered from $1+\sum_{j<i}e_j$ to
$\sum_{j\leq i}e_j$) and the restriction of $f$ to this component has
degree $d_i$.  This variety can also be defined inductively as
follows.  If $r=0$, then $M_{d_0,e_0} = \Mb_{0,e_0}(X,d_0)$.
Otherwise it follows from \cite[Prop.~3.6]{buch.chaput.ea:finiteness}
that $M_{\bd,\be}$ is the product over $X$ of the maps $\ev_{|\be'|} :
M_{\bd',\be'} \to X$ and $\ev_1 : M_{d_r,e_r+1} \to X$, where $\bd' =
(d_0, \dots, d_{r-1})$ and $\be' = (e_0, \dots, e_{r-2}, e_{r-1}+1)$.
Set $\cZ_{\bd,\be} = \ev(M_{\bd,\be}) \subset X^{|\be|}$.  Given
subvarieties $\Omega_1,\dots,\Omega_m$ of $X$ with $m \leq |\be|$,
define a boundary Gromov-Witten variety by
$M_{\bd,\be}(\Omega_1,\dots,\Omega_m) = \bigcap_{i=1}^m
\ev_i^{-1}(\Omega_i) \subset M_{\bd,\be}$.  We also write
$\Gamma_{\bd,\be}(\Omega_1,\dots,\Omega_m) =
\ev_{|\be|}(M_{\bd,\be}(\Omega_1,\dots,\Omega_m)) \subset X$.  If no
sequence $\be$ is specified, we will use $\be=(3)$ when $r=0$ and
$\be=(2,0,\dots,0,1)$ when $r>0$.  This convention will be used only
when $d_i \neq 0$ for $i>0$.  For this reason the sequence $\bd =
(d_0,\dots,d_r)$ will be called a {\em stable sequence of degrees\/}
if $d_i \neq 0$ for $i>0$.

An irreducible variety $Y$ has {\em rational singularities\/} if there
exists a desingularization $\pi : \wt Y \to Y$ such that $\pi_*
\cO_{\wt Y} = \cO_Y$ and $R^i \pi_* \cO_{\wt Y} = 0$ for all $i>0$.
An arbitrary variety has rational singularities if its irreducible
components have rational singularities, are disjoint, and have the
same dimension.  We need the following result from
\cite[Lemma~3]{brion:positivity}.

\begin{lemma}[Brion]\label{lem:ratsingfib}
  Let $Z$ and $S$ be varieties and let $\pi : Z \to S$ be a morphism.
  If $Z$ has rational singularities, then the same holds for the
  general fibers of $\pi$.
\end{lemma}

A morphism $f : Y \to Z$ of varieties is a {\em locally trivial
  fibration\/} if each point $z \in Z$ has an open neighborhood $U
\subset Z$ such that $f^{-1}(U) \cong U \times f^{-1}(z)$ and $f$ is
the projection to the first factor.  The following result is obtained
by combining Propositions 2.2 and 2.3 in
\cite{buch.chaput.ea:finiteness}.

\begin{prop}\label{prop:loctriv}
  Let $B \subset G$ be a Borel subgroup, let $Y$ be a $B$-variety, let
  $\Omega \subset X$ be a $B$-stable Schubert variety, and let $f : Y
  \to \Omega$ be a dominant $B$-equivariant map.  Then $f$ is a
  locally trivial fibration over the dense open $B$-orbit
  $\Omega^\circ \subset \Omega$.
\end{prop}

It was proved in \cite[Prop.~3.7]{buch.chaput.ea:finiteness} that
$M_{\bd,\be}$ is unirational and has rational singularities.
Lemma~\ref{lem:ratsingfib} therefore implies that
$M_{\bd,\be}(x_1,\dots,x_m)$ has rational singularities for all points
$(x_1,\dots,x_m)$ in a dense open subset of $(\ev_1\times \dots \times
\ev_m)(M_{\bd,\be}) \subset X^m$.  Proposition~\ref{prop:loctriv}
applied to the map $\ev_1 : M_{\bd,\be} \to X$ shows that
$M_{\bd,\be}(x)$ is unirational for all points $x \in X$.  Finally,
\cite[Lemma~3.9(a)]{buch.chaput.ea:finiteness} states that the variety
$\cZ_{d,2} = \ev(M_{d,2}) \subset X^2$ is rational and has rational
singularities for any effective degree $d \in H_2(X;\Z)$,

\ignore{
\begin{prop}\label{prop:rat2gw}
  Let $\bd=(d_0,d_1,\dots,d_r)$ be a stable sequence of degrees.  Then
  $M_\bd(x,y)$ is unirational for all points $(x,y)$ in a dense open
  subset of $\cZ_{d_0,2}$.
\end{prop}
\begin{proof}
  Let $\Omega = \Gamma_{d_0}(1.P) \subset X$ be the set of points that
  can be connected to $1.P$ by a rational curve of degree $d_0$.
  Since $M_{d_0}(1.P)$ is irreducible and $P$-stable, it follows that
  $\Omega$ is a $P$-stable Schubert variety.  Let $U \subset \Omega$
  be the dense open $P$-orbit.  It follows from
  Proposition~\ref{prop:loctriv} that $\ev_2 : M_\bd(1.P) \to \Omega$
  is a locally trivial fibration over $U$.  Since $M_\bd(1.P)$ is
  unirational, this implies that $M_\bd(1.P,x)$ is unirational for all
  $x \in U$.  Finally notice that $\cZ_{d_0,2} = G \times^P \Omega =
  (G \times \Omega)/P$, where $P$ acts by $(g,x).p = (gp,p^{-1}.x)$,
  and $M_\bd(x,y)$ is unirational for all points $(x,y)$ in the dense
  open subset $G \times^P U \subset \cZ_{d_0,2}$.
\end{proof}
}

\begin{prop}\label{prop:rat2gw}
  The variety $M_{\bd,\be}(x,y)$ is unirational for all points $(x,y)$
  in a dense open subset of the image $(\ev_1\times\ev_2)(M_{\bd,\be})
  \subset X^2$.
\end{prop}
\begin{proof}
  Set $\Omega = \ev_2(M_{\bd,\be}(1.P)) \subset X$.  Since
  $M_{\bd,\be}(1.P)$ is irreducible and $P$-stable, it follows that
  $\Omega$ is a $P$-stable Schubert variety.  Let $U \subset \Omega$
  be the dense open $P$-orbit.  It follows from
  Proposition~\ref{prop:loctriv} that $\ev_2 : M_{\bd,\be}(1.P) \to
  \Omega$ is a locally trivial fibration over $U$.  Since
  $M_{\bd,\be}(1.P)$ is unirational, this implies that
  $M_{\bd,\be}(1.P,x)$ is unirational for all $x \in U$.  Finally
  notice that $(\ev_1\times\ev_2)(M_{\bd,\be}) = G \times^P \Omega =
  (G \times \Omega)/P$, where $P$ acts by $(g,x).p = (gp,p^{-1}.x)$,
  and $M_\bd(x,y)$ is unirational for all points $(x,y)$ in the dense
  open subset $G \times^P U \subset G \times^P \Omega$.
\end{proof}

\begin{remark}
  It is proved in \cite[Lemma~15.8]{jong.he.ea:families} that, if $\bd
  = (1^d) = (1,1,\dots,1)$ with $d$ large, $\be = (1,0^{d-2},1)$, and
  $\Pic(X)=\Z$, then the general fibers of $\ev : M_{\bd,\be} \to X^2$
  are rationally connected.  This also follows from
  Proposition~\ref{prop:rat2gw}.  A more general statement is proved
  in \cite[Prop.~3.2]{buch.chaput.ea:finiteness}.
\end{remark}

\section{Rationally connected Gromov-Witten varieties}

An algebraic variety $X$ is {\em rationally connected\/} if two
general points $x,y \in X$ can be joined by a rational curve, i.e.\
both $x$ and $y$ belong to the image of some morphism $\bP^1 \to X$.
We need the following fundamental result from
\cite{graber.harris.ea:families}.

\begin{thm}[Graber, Harris, Starr]\label{thm:ratconn}
  Let $f : X \to Y$ be any dominant morphism of complete irreducible
  complex varieties.  If $Y$ and the general fibers of $f$ are
  rationally connected, then $X$ is rationally connected.
\end{thm}

We assume from now on that $X = G/P$ is defined by a maximal parabolic
subgroup $P \subset G$.  Then we have $H_2(X;\Z) = \Z$, so the degree
of a curve in $X$ can be identified with an integer.  We will further
assume that the three-point Gromov-Witten varieties of $X$ of
sufficiently high degree are rationally connected.  More precisely,
assume that there exists an integer $\drc$ such that $M_d(x,y,z)$ is
rationally connected for all $d \geq \drc$ and all points $(x,y,z)$ in
a dense open subset $U_d \subset X^3$.

For $n \geq 2$ we set $d_X(n) = \min\{d \in \N \mid \cZ_{d,n}=X^n\}$.
This is the smallest integer such that, given $n$ arbitrary points in
$X$, there exists a curve of degree $d_X(n)$ through all $n$ points.
Finally we set $\dcl = \min\{d \in \N \mid \cZ_{(1^d),(1,0^{d-2},1)} =
X^2 \}$, where $(1^d) = (1,1,\dots,1)$ denotes a sequence of $d$ ones.
This is the smallest length of a chain of lines connecting two general
points in $X$.  Notice that $d_X(3) \leq \drc$ and $d_X(2) \leq \dcl$.

\begin{thm}\label{thm:gw3ratconn}
  Let $\bd = (d_0,d_1,\dots,d_r)$ be a stable sequence of degrees such
  that $|\bd| \geq \drc + \dcl - 1$.  Then we have $\cZ_\bd =
  \cZ_{d_0,2} \times X$, and $M_\bd(x,y,z)$ is rationally connected
  for all points $(x,y,z)$ in a dense open subset of $\cZ_{\bd}$.
\end{thm}
\begin{proof}
  Set $\bd' = (d_1,\dots,d_r)$ and $\be' = (1,0,\dots,0,1) \in \N^r$.
  We then have $d_0 \geq \drc$ or $|\bd'| \geq \dcl$.  Assume first
  that $|\bd'| \geq \dcl$.  It then follows from the definition of
  $\dcl$ that $\cZ_\bd = \cZ_{d_0,2} \times X$.  Let $X^\circ = P
  w_0.P \subset X$ be the open $P$-orbit.  By
  Proposition~\ref{prop:rat2gw} and Lemma~\ref{lem:ratsingfib} we may
  choose a dense open subset $U \subset \cZ_{d_0,2}$ such that, for
  all points $(x,y) \in U$ we have that $M_{d_0}(x,y)$ is unirational,
  $\Gamma_{d_0}(x,y) \cap X^\circ \neq \emptyset$, and
  $M_\bd(x,y,1.P)$ has rational singularities.  Let $(x,y) \in U$.  We
  will show that $M_\bd(x,y,1.P)$ is rationally connected.  Let $p :
  M_\bd(x,y,1.P) \to M_{d_0}(x,y)$ be the projection.  Then the fibers
  of $p$ are given by $p^{-1}(f) = M_{\bd',\be'}(\ev_3(f),1.P)$.
  Since the morphism $\ev_1 : M_{\bd',\be'}(X,1.P) \to X$ is
  surjective and $P$-equivariant, Proposition~\ref{prop:loctriv}
  implies that this map is locally trivial over $X^\circ$.  Since
  $M_{\bd',\be'}(X,1.P)$ is unirational, we deduce that
  $M_{\bd',\be'}(z',1.P)$ is unirational for all $z' \in X^\circ$.
  This implies that $p^{-1}(f)$ is unirational for all $f \in
  M_{d_0}(x,y,X^\circ)$, which is a dense open subset of
  $M_{d_0}(x,y)$ by choice of $U$.  Since the general fibers of $p$
  are connected, it follows from Stein factorization that all fibers
  of $p$ are connected.  Therefore $M_\bd(x,y,1.P)$ is connected.
  Since this variety also has rational singularities, we deduce that
  $M_\bd(x,y,1.P)$ is irreducible.  Finally, Theorem~\ref{thm:ratconn}
  applied to the map $p : M_\bd(x,y,1.P) \to M_{d_0}(x,y)$ shows that
  $M_\bd(x,y,1.P)$ is rationally connected.
  
  Assume now that $d_0 \geq \drc$.  In this case we have $\cZ_{\bd} =
  X^3$.  Let $U \subset X^3$ be a dense open subset such that
  $M_\bd(x,y,z)$ has rational singularities and $M_{d_0}(x,y,z)$ is
  rationally connected and has rational singularities for all $(x,y,z)
  \in U$.  We will show that $M_\bd(x,y,z)$ is rationally connected
  for all $(x,y,z) \in U$.  Let $q : M_\bd(x,y,z) \to
  M_{\bd',\be'}(X,z)$ be the projection and set $U_{x,y} = \{z' \in X
  \mid (x,y,z') \in U\}$.  Then $q^{-1}(f) = M_{d_0}(x,y,\ev_1(f))$ is
  rationally connected for all $f \in M_{\bd',\be'}(U_{x,y},z)$, which
  is a dense open subset of $M_{\bd',\be'}(X,z)$.  Since the general
  fibers of $q$ are connected, it follows by Stein factorization that
  all fibers of $q$ are connected, hence $M_\bd(x,y,z)$ is connected.
  Since $M_\bd(x,y,z)$ has rational singularities, this implies that
  $M_\bd(x,y,z)$ is irreducible.  Theorem~\ref{thm:ratconn} applied to
  the map $q : M_\bd(x,y,z) \to M_{\bd',\be'}(X,z)$ finally shows that
  $M_\bd(x,y,z)$ is rationally connected, as required.
\end{proof}

\section{Quantum $K$-theory}

Let $K(X)$ denote the Grothendieck ring of algebraic vector bundles on
$X$.  An introduction to this ring can be found in e.g.\
\cite[\S3.3]{brion:lectures}.  For each effective degree $d \in
H_2(X;\Z)$ we define a class $\Phi_d \in K(X^3)$ by
\[
\Phi_d = \sum_{\bd=(d_0,\dots,d_r)} (-1)^r \ev_*[\cO_{M_\bd}] \,,
\]
where the sum is over all stable sequences of degrees $\bd$ such that
$|\bd| = d$, and $\ev : M_\bd \to X^3$ is the evaluation map.  Let
$\pi_i : X^3 \to X$ be the projection to the $i$-th factor.  For $\al,
\beta \in K(X)$ we set $(\al \star \beta)_d =
{\pi_3}_*(\pi_1^*(\al)\cdot \pi_2^*(\beta) \cdot \Phi_d) \in K(X)$.
The quantum $K$-theory ring of $X$ is an algebra over $\Z\llbracket
q\rrbracket$, which as a $\Z\llbracket q\rrbracket$-module is given by
$\QK(X) = K(X) \otimes_\Z \Z\llbracket q \rrbracket$.  The
multiplicative structure of $\QK(X)$ is defined by
\[
\al \star \beta = \sum_d (\al \star \beta)_d\, q^d
\]
for all classes $\al,\beta \in K(X)$, where the sum is over all
effective degrees $d$.  A theorem of Givental \cite{givental:wdvv}
states that $\QK(X)$ is an associative ring.  We note that the
definition of $\QK(X)$ given here is different from Givental's
original construction; the equivalence of the two definitions follows
from \cite[Lemma~5.1]{buch.chaput.ea:finiteness}.\footnote{Lemma~5.1
  in \cite{buch.chaput.ea:finiteness} is stated only for cominuscule
  varieties, but its proof works verbatim for any projective rational
  homogeneous space $X$.}

We need the following Gysin formula from
\cite[Thm.~3.1]{buch.mihalcea:quantum} (see also
\cite[Prop.~5.2]{buch.chaput.ea:finiteness} for the stated version.)

\begin{prop}
  Let $f : X \to Y$ be a surjective morphism of projective varieties
  with rational singularities.  If the general fibers of $f$ are
  rationally connected, then $f_*[\cO_X] = [\cO_Y] \in K(Y)$.
\end{prop}

\begin{cor}\label{cor:pushMd}
  Let $\bd = (d_0,\dots,d_r)$ be a stable sequence of degrees such
  that $|\bd| \geq \drc + \dcl - 1$.  Then we have $\ev_*[\cO_{M_\bd}]
  = [\cO_{\cZ_\bd}] \in K(X^3)$.
\end{cor}
\begin{proof}
  This holds because $\cZ_\bd = \cZ_{d_0,2} \times X$ has rational
  singularities \cite[Lemma~3.9]{buch.chaput.ea:finiteness}, the
  general fibers of the map $\ev : M_\bd \to \cZ_\bd$ are rationally
  connected by Theorem~\ref{thm:gw3ratconn}, and $M_\bd$ has rational
  singularities by \cite[Prop.~3.7]{buch.chaput.ea:finiteness}.
\end{proof}

Theorem~\ref{thm:main} is equivalent to the following result.

\begin{thm}\label{thm:Phid}
  We have $\Phi_d = 0$ for all $d \geq \drc + \dcl$.
\end{thm}
\begin{proof}
  It follows from Corollary~\ref{cor:pushMd} that, for $d \geq
  \drc+\dcl$ we have
  \[
  \Phi_d = \sum_{\bd=(d_0,\dots,d_r)} (-1)^r [\cO_{\cZ_\bd}] \ \in
  K(X^3) \,,
  \]
  where the sum is over all stable sequences of degrees $\bd$ with
  $|\bd|=d$.  Since $\cZ_\bd = \cZ_{d_0,2}\times X$, the terms of this
  sum depend only on $d_0$.  Since $d \geq \dcl+\drc > d_X(2)$, it
  follows that $\cZ_{(d)} = \cZ_{(d-1,1)} = X^3$, so the contributions
  from the sequences $\bd=(d)$ and $\bd=(d-1,1)$ cancel each other
  out.  Now let $0 \leq d' \leq d-2$.  For each $r$ with $1 \leq r
  \leq d-d'$, there are exactly $\binom{d-d'-1}{r-1}$ sequences $\bd$
  in the sum for which $d_0=d'$ and the length of $\bd$ is $r+1$.
  Since $\sum_{r=1}^{d-d'} (-1)^r \binom{d-d'-1}{r-1} = 0$, it follows
  that the corresponding terms cancel each other out.  It follows that
  $\Phi_d=0$, as claimed.
\end{proof}

\begin{remark}
  Theorem~\ref{thm:Phid} is true also for the equivariant $K$-theory
  ring $\QK_T(X)$ with the same proof.
\end{remark}

\begin{remark}
  If $X$ is not the projective line, then the proof of
  Theorem~\ref{thm:Phid} shows that $\Phi_d=0$ for all $d \geq \drc +
  \dcl - 1$.  It would be interesting to determine the maximal value
  of $d$ for which $\Phi_d \neq 0$.  If $X$ is a cominuscule variety,
  then this number is equal to $d_X(2)$, hence the maximal power of
  $q$ that appears in products in the quantum $K$-theory ring of $X$
  is equal to the maximal power that appears in the quantum cohomology
  ring \cite{buch.chaput.ea:finiteness}.
\end{remark}


\providecommand{\bysame}{\leavevmode\hbox to3em{\hrulefill}\thinspace}
\providecommand{\MR}{\relax\ifhmode\unskip\space\fi MR }
\providecommand{\MRhref}[2]{%
  \href{http://www.ams.org/mathscinet-getitem?mr=#1}{#2}
}
\providecommand{\href}[2]{#2}

\end{document}